\documentclass[a4paper,12pt]{article}
\usepackage[utf8x]{inputenc}
\usepackage[english]{babel}
\usepackage{graphicx}
\usepackage{amscd}
\usepackage{amsfonts}
\usepackage{latexsym}
\usepackage{amssymb,amsmath}
\usepackage[usenames]{color}

\usepackage{amscd}
\usepackage{psfrag}

\usepackage{amsthm}

\theoremstyle{plain}
\newtheorem{teor}{Theorem}[section]
\newtheorem{claim}{Claim}[section]
\newtheorem{cor}[teor]{Corollary}%[section]
\newtheorem{prop}[teor]{Proposition}
\newtheorem{lemma}{Lemma}[section]

\newcommand{\THM}{\begin{teor}}
\newcommand{\REFTHM}[1]{\begin{teor}\label{#1}}
\newcommand{\ENDTHM}{\end{teor}}
\newcommand{\COR}{\begin{cor}}
\newcommand{\REFCOR}[1]{\begin{cor}\label{#1}}
\newcommand{\ENDCOR}{\end{cor}}
\newcommand{\REFCLAIM}[1]{\begin{claim}\label{#1}}
\newcommand{\ENDCLAIM}{\end{claim}}
\newcommand{\PROP}{\begin{prop}}
\newcommand{\REFPROP}[1]{\begin{prop}\label{#1}}
\newcommand{\ENDPROP}{\end{prop}}
\newcommand{\LEM}{\begin{lemma}}
\newcommand{\REFLEM}[1]{\begin{lemma}\label{#1}}
\newcommand{\ENDLEM}{\end{lemma}}
\newcommand{\REFEQN}[1]{\begin{equation}\label{#1}}
\newcommand{\ENDEQN}{\end{equation}}
\theoremstyle{definition}

  \newcommand*{\mapfromto}[3]{\hbox{\ensuremath{#1 : #2 \longrightarrow #3}}}
\def\chh{{\check{h}}}
\def\chH{{\check{H}}}
\def\chp{{\check{p}}}
\def\chP{{\check{P}}}
\def\M{\mathcal{M}}
\def\X{\mathcal{H}}
\def\F{\mathcal{F}}
\def\T{\mathcal{T}}
\def\OO{\mathcal{O}}
\def\TT{\mathbb{T}}
\def\C{\mathbb{C}}
\def\S{\mathbb{S}}
\def\Sen{{\mathbb{S}^1}}
\def\D{\mathbb{D}}
\def\Dbar{\mathbb{{\overline{D}}}}
\def\R{\mathbb{R}}
\def\P{\mathcal{P}}

\def\Z{\mathbb{Z}}

\def\H{\mathbb{H}}
\def\pp{_{\rho}}
\def\eps{\epsilon}
\def\ga{\gamma}

\def\l0{_{\lambda_0}}
\def\i0{_{\iota_0}}

\def\BV{{\operatorname{BV}}}
\def\d{{\operatorname{d}}}
\def\dx{{\operatorname{dx}}}
\def\e{{\operatorname{e}}}
\def\Om{\Omega}
\newcommand{\Par}{\operatorname{Par}}
\def\sm{\setminus}
\def\whF{{\widehat{\mathcal{F}}}}

\def\whh{{\widehat{h}}}
\def\whp{{\widehat{p}}}
\def\whx{{\widehat{x}}}

\def\whH{{\widehat{H}}}
\def\whP{{\widehat{P}}}

\def\whR{{\widehat{R}}}

\def\wth{{\widetilde{h}}}
\def\wtphi{{\widetilde{\phi}}}

\def\ITMZ{\begin{itemize}}
\def\ENDITMZ{\end{itemize}}
\def\PROOF{\begin{proof}}
\newcommand{\REFPROOF}[1]{\begin{proof}{#1}}
\def\ENDPROOF{\end{proof}}
\def\Mobius{M{\"o}bius}
\newcommand{\lemref}[1]{{Lemma~\ref{#1}}}
\newcommand{\corref}[1]{{Corollary~\ref{#1}}}

\newcommand{\propref}[1]{{Proposition~\ref{#1}}}
\title{On parabolic external maps}
\author{L. Lomonaco, C. Petersen, W. Shen}

\begin{document}

\maketitle

\begin{abstract}
We prove that any $C^{1+BV}$ degree $d\geq 2$ circle covering $h$ 
having all periodic orbits weakly expanding, is conjugate by a $C^{1+BV}$ 
diffeomorphism to a metrically expanding map. 
We use this to connect the space of parabolic external maps (coming from 
the theory of parabolic-like maps) to metrically expanding circle coverings.
\end{abstract}

\section{Introduction}

In this paper we provide a connection between the worlds of real and complex dynamics by 
proving theorems on degree $d\geq 2$ circle coverings which are interesting 
in the world of real dynamics per se and interesting in the world of complex dynamics 
through quasi-conformal surgery.
The main theorem states that any $C^{1+BV}$  degree $d\geq 2$ circle covering $h$
(where $h \in C^{1+BV}$ means $Dh$ is continuous and of bounded variation), 
all of whose periodic orbits are weakly expanding, 
is conjugate in the same smoothness class to a metrically expanding map. 
Here weakly expanding means that for any periodic point $p$ of period $s$ there exists a punctured neighborhood of $p$ on which $Dh^s(x)>1$. 
And metrically expanding means 
$Dh(x)>1$ holds everywhere, except at parabolic points. 
This theorem strengthens a theorem by Ma\~n\'e \cite{Mane} who proved the same 
conclusion holds under the stronger assumption that $h$ is $C^2$ with all periodic points hyperbolic repelling.

The real analytic version of the above theorem, which comes for the same price, 
provides a missing link between the space of parabolic external maps from 
the theory of parabolic-like maps and metrically expanding circle coverings. 
For an enlargement on the theory of parabolic-like maps 
and the role of parabolic external maps in this theory see the introduction to 
Section~\ref{parabolicexternalmaps} and the paper by the first author \cite{L}.

\section{Setting and statement of the results}
Recall that a locally diffeomorphic covering map $h:\S^1\to\S^1$ has degree $d\not=0$ 
if and only if $h(\e^{i2\pi x})$ lifts to $E(x):=\e^{i2\pi x}$ as a diffeomorphism 
$H :\R\to\R$ with $H(x+1) = H(x)+d$. 
It will be convenient to work mostly with $H$ and the induced map 
also denoted $H : \TT \to \TT:= \R/\Z$.
Denote by $\F_d$ the set of all locally diffeomorphic, 
real analytic covering maps $h:\S^1\to\S^1$ of degree $d$. 
And denote by $\Par(h)$ the set of parabolic periodic points for $h$. 
Note that for $0<d$ the multiplier of any parabolic orbit is $1$.

We denote by $\F_d^{1+\BV}$ the set of locally diffeomorphic 
covering maps $h:\S^1\to\S^1$ of degree $d>0$ with $h\in C^{1+\BV}$ (i.e. identifying $\Sen$ with $\TT:=\R/\Z$ and writing $Dh$ for the derivative of $h$, 
the function $Dh$ is continuous and of bounded variation). 
Clearly for every $d\geq1$, $\F_d\subset\F_d^{1+\BV}$ (note that we find a whole range of spaces in between such as 
$\F_d^{\,r+\eps} := \F_d^{1+\BV}\cap C^{r+\eps}$ with $r=2, 3,\ldots$ and $0<\eps\leq 1$, 
where $C^{r+\eps}$ are the maps which have an $\eps$-H{\"o}lder $r$-th derivative. 
In this context $\F_d$ is naturally denoted $\F_d^\omega$).
%We denote by $\F_d^{1+\BV}$ the set of locally diffeomorphic 
%covering maps $h:\S^1\to\S^1$ of degree $d>0$ with derivative continuous and of bounded variation 
% i.e. identifying $\Sen$ with $\TT:=\R/\Z$ and writing $Dh$ for the derivative of $h$, 
% the function $Dh$ is continuous and of bounded variation. 
% Clearly for every $d\geq1$, $\F_d\subset\F_d^{1+\BV}$ (and we find a whole range of spaces in between such as 
% $\F_d^{\,r+\eps} := \F_d^{1+\BV}\cap C^{r+\eps}$ with $r=2, 3,\ldots$ and $0<\eps\leq 1$, 
% where $C^{r+\eps}$ are the maps which have an $\eps$-H{\"o}lder $r$-th derivative. 
% Note that in this context $\F_d$ is naturally denoted $\F_d^\omega$).

% A degree $d$ circle map $h : \S^1 \rightarrow \S^1$ is \textit{topologically expanding}, 
% if every interval eventually expands onto the entire circle or 
% equivalently the map is topologically conjugate to the map $P_d(z)=z^d$.
% We denote by $\T_d^{1+\BV} = 
% \{ h \in \F_d^{1+\BV} | \,\,h~\textrm{is topologically expanding} \}$,
% by $\T_d:=\T_d^{1+\BV}\cap\F_d$ and by $\T_d^1\subset \T_d$ 
% for the real-analytic maps with $\Par(h)$ a singleton.

Consider the class  $\T_d^*\subset\F_d^{1+\BV}$ of maps $h$ 
for which for every periodic point $p$ say of period $s$, 
there is a neighborhood $U(p)$ of $p$ such that 
for all $x\in U(p)\setminus \{p\}$ : $Dh^s(x)>1$. 
This is, $h$ is a degree $d$ locally diffeomorphic covering of the circle,
with continuous derivative of bounded variation, 
and with all periodic points either repelling or parabolic-repelling. 

Finally a map $h\in \F_d^{1+\BV}$ is called {\em metrically expanding} 
if for all $x$ in $\S^1\setminus \Par(h)$, $Dh(x)>1$. 
We shall see that for such maps $\Par(h)$ is a finite set. 
We denote by $\M_d^{1+\BV}\subset \F_d^{1+\BV}$ 
the sub-class of metrically expanding maps,
$\M_d = \F_d\cap\M_d^{1+\BV}$, and by $\M_{d,1}\subset \M_d$ 
the set of real-analytic such $h$ with precisely one parabolic point.

In this paper we will prove the following result: 

\REFTHM{TdisessMd}
For each map $h\in\T_d^*$, the set $\Par(h)$ is finite 
and $h$ is conjugate to a map $\wth\in\M_d^{1+\BV}$ via a $C^{1+\text{BV}}$ diffeomorphism. 

Moreover, if $h$ is $C^{r+\eps}$ for some 
$r=2,3,\ldots,\infty$ and $0<\eps\leq 1$ or $C^\omega$
we can take the conjugacy map to be $C^{r +\eps}$ respective $C^\omega$. 
\ENDTHM
The proof of this Theorem is the core objective of Section~\ref{viceversa}, 
where it is stated as 
Corollary~\ref{mainBV}. 

A degree $d$ circle map $h : \S^1 \rightarrow \S^1$ is \textit{topologically expanding}, 
if every interval eventually expands onto the entire circle or  equivalently the map is topologically conjugate to the map $P_d(z)=z^d$.
We denote by $\T_d^{1+\BV} = 
\{ h \in \F_d^{1+\BV} | \,\,h~\textrm{is topologically expanding} \}$,
by $\T_d:=\T_d^{1+\BV}\cap\F_d$ and by $\T_{d,1}\subset \T_d$ 
the set of topologically expanding real-analytic maps with $\Par(h)$ a singleton.
Theorem \ref{TdisessMd} implies the following:
\REFCOR{TdstarequalsTdoneplusBV}
We have $\T_d^*=\T_d^{1+\BV}$ and 
for all $h\in\M_d^{1+\BV}$ the set $\Par(h)$ is a finite set.
\ENDCOR

\noindent Denote by $\tau$ the reflection, 
both in $\Sen$ ($\tau(z) := 1/\overline{z}$) and in $\R$ ($\tau(z) = \overline{z}$).

\noindent A  \textit{parabolic external map} is a map $h\in\F_d$ with the following properties:
\ITMZ
\item $h: \S^1 \rightarrow \S^1$ is a degree $d\ge 2$ real-analytic covering of the unit circle, with a finite set $Par(h)$ of parabolic points $p$
of multiplier $1$, 
\item the map $h$ extends to a holomorphic covering map $h : W' \to W$ of degree $d$, 
where $W'$, $W$ are reflection symmetric annular neighborhoods of $\Sen$. 
%each bounded by a pair of $C^1$ Jordan curves. 
We write $W_+:= W\sm\Dbar$, 
and $W'_+ := W'\sm\Dbar$, 
\item for each $p \in Par(h)$ there exists a \emph{dividing arc} $\ga_p$ satisfying :
\begin{itemize}
\item 
$p\in\gamma_p \subset \overline W\sm\D$ and $\gamma_p$ is smooth except at $p$, 
\item
$\gamma_p \cap \gamma_{p'} = \emptyset$ for $p \neq p'$, 
\item $h:\gamma_p \cap W' \rightarrow \gamma_{h(p)}$ is a diffeomorphism,
\item $\ga_p$ divides $W$ and $W'$ into $\Omega_p,\,\Delta_p$ and $\Omega'_p,\,\,\Delta'_p$ respectively, all connected, and such that 
 $h: \Delta'_p \rightarrow \Delta_{h(p)}$ is an isomorphism and 
 $\D\cup\Omega'_p \subset \D\cup\Omega_p$,
 \item calling $\Omega= \bigcap_p \Omega_p$ and $\Omega'= \bigcap_p \Omega'_p$, we have $\Omega'\cup\D \subset \subset W\cup\D$.
\end{itemize}
\ENDITMZ
We denote by $\P_d\subset \F_d$ the set of parabolic external maps, and by $\P_{d,1}\subset\P_d$ the set of $h\in \P_d$ for which $Par(h)$ 
is a singleton $z_0$. 
To emphasize the geometric properties of maps $h\in\P_d$ we shall also write 
$(h, W', W, \ga)$ for such maps, where $\gamma= \bigcup_{p} \gamma_p$, 
though neither the domain, range or dividing arcs are unique or in any way canonical. 
An external map for any parabolic-like map belongs to $\P_d$ (see Section \ref{extmap}). Note that the set $\P_d$ is invariant under conjugacy by a real analytic diffeomorphism
this is, for any $h\in\P_d$ and $\phi\in\F_1$ : $\phi\circ h\circ\phi^{-1}\in\P_d$. 
It is easy to see that $\P_d\subset\T_d$ (see \propref{PdinTdstarcapFd}). 
In particular $\P_{d,1}\subset\T_{d,1}$ and so any two maps $h_1, h_2 \in \P_{d,1}$ 
are topologically conjugate by a unique orientation preserving homeomorphism 
sending parabolic point to parabolic point.

Set $h_d(z)= \frac{z^d+(d-1)/(d+1)}{(d-1)z^d/(d+1)+1}$,
and define $\X_{d,1} =\{ h \in \F_d | \,\,h \sim_{qs} h_d \}$
(where $h \sim_{qs} h_d $ means that $h$ is quasi-symmetrically conjugate to the map $h_d$).
It is rather easy to see that $h_d\in\P_{d,1}$, 
see Lemma~\ref{hdinPd}. 
Moreover clearly $h_d\in\T_d$ so that $\X_{d,1}\subseteq\T_{d,1}$. 
In Section \ref{proof} we prove:

\REFPROP{topgivesqs}
Suppose $h_1, h_2\in\P_d$ are topologically conjugate by an orientation preserving 
homeomorphism $\phi$, which preserves parabolic points. 
Then $\phi$ is quasi-symmetric. 
In particular $\P_{d,1}\subseteq \X_{d,1}\subseteq\T_{d,1}$.
\ENDPROP
Let $\whF_d := \F_d/\F_1$ denote the set of conjugacy classes of maps in $\F_d$ 
under real analytic diffeomorphisms, 
and call $\pi_d: \F_d \rightarrow \whF_d$ the natural projection. 
As a consequence of the above we have (see also page~\pageref{proofofth}):
\REFTHM{th}
For every $d\geq 2$ we have 
$$
\M_d\subset\P_d=\T_d\qquad\textrm{and}\qquad\M_{d,1}\subset\P_{d,1}=\X_{d,1}=\T_{d,1}
$$
Moreover 
\begin{align*}
\pi_d(\M_d)=\pi_d(\P_d)&=\pi_d(\T_d)\qquad\textrm{and}\\
\pi_d(\M_{d,1})=\pi_d(\P_{d,1})&=\pi_d(\X_{d,1})=\pi_d(\T_{d,1}).
\end{align*}
\ENDTHM

\textbf{Aknowledgment.} 
The first author would like to thank the Fapesp for support via the process 2013/20480-7.
\section{Topologically expanding self coverings of the unit circle}\label{viceversa}
Recall that for each integer $d\ge 2$, the set $\mathcal{T}_d^*$ 
denotes the collection of all orientation preserving covering maps 
$h: \S^1\to \S^1$ with the following properties:
\begin{enumerate}
\item $h$ has degree $d$;
\item $h$ is a $C^1$ local diffeomorphism and the derivative $Dh$ has bounded variation;
\item If $p$ is a periodic point of $h$ with period $s$, then there is a neighborhood $U(p)$ of $p$ such that
 $Dh^s(x)>1$ holds for all $x\in U(p)\setminus \{p\}$.
\end{enumerate}

\begin{teor}\label{thm:circlemap}
For each map $h\in\mathcal{T}_d^*$, $\Par(h)$ is finite.
Moreover,
there exists a positive integer $N$ and a real analytic function $\rho: \S^1 \to \R^+$ such that
$$|Dh^N(x)|_\rho:=\frac{\rho(h^N(x))}{\rho(x)} Dh^N(x) >1$$ holds for all $x\in \S^1\setminus \Par(h)$.
\end{teor}

In particular, the theorem claims that a map $h\in\mathcal{T}_d^*$ without neutral cycles is uniformly expanding on the whole phase space $\S^1$, a result proved by Ma\~n\'e \cite{Mane} under a stronger assumption that $h$ is $C^2$. Some partial result on the validity of Ma\~n\'e's theorem under the $C^{1+\text{BV}}$ condition was obtained in \cite{MJ}.

Recall that a map $h\in \mathcal{T}_d^*$ is called {\em metrically expanding} if $Dh(x)>1$ holds for $x\in \S^1\setminus \Par(h)$. 

\REFCOR{mainBV} Each map $h\in\mathcal{T}_d^*$ is conjugate to a metrically expanding map via a $C^{1+\text{BV}}$ diffeomorphism. Moreover, if $h$ is $C^{r+\eps}$, 
$r=2,3,\ldots,\infty, \omega$, $0<\eps\leq 1$ 
we can take the conjugacy map to be $C^{r+\eps}$.
\ENDCOR
\begin{proof}
Let $\rho$ and $N$ be given by Theorem~\ref{thm:circlemap} and set
$$\rho_*(x)=\sum_{j=0}^{N-1} \rho(h^j(x)) Dh^j(x)$$
which is a continuous function with bounded variation. 
Then a computation shows that
$$|Dh(x)|_{\rho_*}=Dh(x)\cdot \frac{\rho_*(h(x))}{\rho_*(x)}=
 \frac{Dh^N(x) \rho(h^N(x))+ \sum_{j=1}^{N-1}Dh^j(x)\rho(h^j(x))}{\rho(x) + \sum_{j=1}^{N-1}Dh^j(x)\rho(h^j(x))} $$
which is strictly greater than one for $x\in\S^1\setminus \Par(h)$.

Let us complete the proof. Identify $\S^1$ with $\R/\Z$ via $e^{i2\pi x}\mapsto x\mod 1$.
Then the map
$$\phi(x):=\int_0^x  C \rho_*
dx= C \int_0^x \rho_* dx, \mbox{  with  }\frac{1}{C}= \int_0^1 \rho_* dx$$
defines a $C^{1+\text{BV}}$ diffeomorphism of $\S^1$, and setting $g:= \phi \circ h \circ \phi^{-1}$, we obtain that
$Dg(\phi(x))=|Dh(x)|_{\rho_*}>1$ for all $x\in \S^1\setminus \Par(h)$.

Clearly, if $h$ is $C^r$ then so is $\phi$.
\end{proof}

The rest of the paper is devoted to the proof of Theorem~\ref{thm:circlemap}.
The condition that $Dh$ has bounded variation is used to control the distortion.
Recall that the distortion of $h$ on an interval $J \subset \S^1$ is defined as
$$Dist (h, J)= \sup_{x,y \in J} \log \frac{|Dh(x)|}{|Dh(y)|}.$$
\begin{lemma}\label{dist}
There exists a $C_0>0$ such that, for any interval $J \subset \S^1$ and $n\ge 1$, if
$J,\,h(J),..,h^{n-1}(J)$ are intervals with pairwise disjoint interiors, then
$$Dist (h^n , J ) \leq C_0. $$
\end{lemma}

\begin{proof}
Since $h$ is a $C^1$ covering and $Dh$ has bounded variation, $\log Dh$ also has bounded variation. For each $x,y\in J$,
$$\left|\log \frac{Dh^n(x)}{Dh^n(y)}\right|\le \sum_{i=0}^{n-1} \left|\log \frac{Dh(h^i(x))}{Dh(h^i(y))}\right|
\le \sum_{i=0}^{n-1} \text{Var}(\log Dh, h^i(J))$$
is bounded from above by the total variation of $\log Dh$.
\end{proof}

\subsection{Proof of Theorem~\ref{thm:circlemap}}
The main step is to prove that a map $h\in\mathcal{T}_d^*$ has the following expanding properties.
\begin{prop}\label{prop:expanding}
For each $h\in\mathcal{T}_d^*$ the following properties hold:
\begin{enumerate}
\item[(a)] $\Par(h)$ is a finite set.
\item[(b)] There exists a constant $K_0>0$ such that $Dh^k(x)\ge K_0$ holds for each $x\in \S^1$ and $k\ge 1$.
\item[(c)] For each $x\not\in \bigcup_{k=0}^\infty h^{-k}(\Par(h))$, $Dh^n(x)\to\infty$ as $n\to\infty$.
\item[(d)]  Let $p$ be a fixed point and let $\delta_0>0$, $K>0$ be constants. 
Then there exists $\delta=\delta(p,\delta_0,K)>0$ such that
if $$d(x, p)<\delta\text{ and }\max_{j=1}^k d(h^{j}(x),p)\ge \delta_0,$$ then $Dh^{k}(x)\ge K.$
\item[(e)] For any $K>0$, there exists a positive integer $n_0$ such that for each $n> n_0$ and $x\in h^{-n}(\Par(h))\setminus h^{-n+1}(\Par(h))$, we have $Dh^n(x)\geq K.$
\end{enumerate}
\end{prop}

Assuming the proposition for the moment, let us complete the proof of Theorem~\ref{thm:circlemap}.
Replacing $h$ by an iterate if necessary, we may assume that {\em all points in $\Par(h)$ are fixed points} (since $\Par(h)$ is finite). We say that a function $\rho:\S^1\to (0,\infty)$ is {\em admissible} if the following properties are satisfied:
\begin{enumerate}
 \item[(A1)] there is $\delta_0>0$ such that whenever $x\in B(p,\delta_0)\setminus \{p\}$ for some $p\in\Par(h)$, we have
 $\rho(h(x))>\rho(x)$;
 \item[(A2)] for any $x \in \S^1 \setminus \Par(h)$ and $s\ge 1$ with $h^s(x)\in \Par(h)$, we have
$$|Dh^s(x)|_{\rho}\geq 2. $$
\end{enumerate}

\begin{lemma} There exists a real analytic admissible function $\rho$.
\end{lemma}
\begin{proof} Let $X_0=\Par(h)$ and $X_k=h^{-k}(\Par(h))\setminus h^{-k+1}(\Par(h))$ for each $k\ge 1$. By Proposition~\ref{prop:expanding} (e), there exists $n_0$ such that $Dh^n(x)\ge 4$ holds for $x\in X_n$, $n>n_0$.
Let $\rho_0=\min\{Df^k(x): x\in X_k \text{ for some } k=1,2,\ldots, n_0\}$.
Let $\pi:\R\to \S^1$ be the universal covering $\pi(t)=e^{2\pi it}$. Let $\widehat{\rho}: \R\to (0,\infty)$ be a real analytic function of period $1$ with 
the following properties:
\begin{enumerate}
\item [(i)] $\widehat{\rho}(\hat{p})=1$, $\widehat{\rho}'(\hat{p})=0$ and $\widehat{\rho}''(\hat{p})>0$ for each $\hat{p}\in \pi^{-1}(\Par(h))$;
\item [(ii)] $\widehat{\rho}(\hat{x}) < \rho_0/2$ holds for each 
$\hat{x}\in \pi^{-1}(X_1\cup X_2\cup\cdots\cup X_{n_0})$; 
\item [(iii)] $0<\widehat{\rho}(\hat{x})<2$ for all $\widehat{x}\in \R$.
\end{enumerate}
It is easy to see that there is a smooth function $\check\rho$ satisfying all the requirements. 
To get a real analytic one, choose $\epsilon>0$ such that (ii) holds for $\check\rho$ on a $2\epsilon$-neighbourhood of $\pi^{-1}(X_1\cup X_2\cup\cdots\cup X_{n_0})$ and 
$\check\rho''(x)>0$ on a $2\epsilon$-neighbourhood of $\pi^{-1}(\Par(h))$. 
Write $\pi^{-1}(\Par(h))\cap[0,1[ := \{\hat{p}_1 < \ldots < \hat{p}_n\}$ and 
let $\delta>0$ be given by Lemma~\ref{adjustment} below. 
And let $\widetilde\rho$ be a partial sum of the Fourier series of $\check\rho$ satisfying 
$\widetilde\rho'(y_j)=0$ for some $y_j$ with $|y_j-\hat{p}_j|<\delta$ for each $j$, 
$\widetilde\rho''(x)>0$ on a $\epsilon$-neighbourhood of $\pi^{-1}(\Par(h))$ and 
$\widetilde{\rho} < \rho_0/2$ on an $\epsilon$-neighbourhood of $\pi^{-1}(X_1\cup X_2\cup\cdots\cup X_{n_0})$. 
Let $\phi$ be the corresponding real-analytic diffeomorphism 
given by Lemma~\ref{adjustment}. 
Then $\widehat\rho = \widetilde\rho\circ\phi$ is the desired real-analytic function.

The function $\widehat{\rho}$ induces a function $\rho:\S^1\to \R$ by the formula $\rho(e^{2\pi t})=\widehat{\rho}(t).$
The property (A1) follows from (i) immediately. Let us check the property (A2). 
Of course it suffices to show $|Dh^n(x)|_{\rho}\ge 2$ for each $x\in X_n$, $n\ge 1$. 
If $n\le n_0$, then $Dh^n(x)\ge \rho_0$, $\rho(x)\le \rho_0/2$ and $\rho(h^n(x))=1$, 
hence $|Dh^n(x)|_{\rho}\ge 2$. 
If $n>n_0$, then $Dh^n(x)\ge 4$, $\rho(x)<2$ and $\rho(h^n(x))=1$, 
hence again $|Dh^n(x)|_\rho\ge 2$.
\end{proof}

\begin{lemma}\label{adjustment}
Given $\epsilon>0$ and $n\ge 1$ distinct $x_1<\ldots < x_n < x_1 + 1$ there 
exists $\delta > 0$ such that for any set of $n$ points $y_1, \ldots , y_n$ with 
$|y_j-x_j| < \delta$ for each $j, 1 \le j \le n$ there exists a real-analytic diffeomorphism 
$\phi : \R \to \R$, satisfying for all $x\in\R$: $\phi(x+1) = \phi(x)+1$, $|\phi(x)-x| < \epsilon$, 
and $|\phi'(x)-1|<\epsilon$ and for each $j$ : $\phi(x_j) = y_j$.
\end{lemma}
\begin{proof}
If $n=1$ set $\delta=\epsilon$ and $\phi(x) = x + y_1 - x_1$. Otherwise set 
$$
m = \min \{(x_2-x_1), \ldots (x_n-x_{n-1}), (1+x_1-x_n)\}
$$
and define $g_j(x) := \sin^2(\pi(x-x_j))$ for each $j, 1\le j \le n$. 
Then $g_j$ is $1$-periodic, $0 \le g_j(x) \le 1$ for all x with $g_j(x)=0$ only at $x_j$ 
and the absolute value of $g_j'(x) = \pi\sin(2\pi(x-x_j))$ is bounded by $\pi$.
Set 
$$
G_j(x) := \prod_{i, i\not=j} g_i(x)
$$
So that $0\le G_j(x) \le 1$, $G_j(x_i) = 0$ for $i\not=j$, $|G_j'(x)| \le \pi(n-1)$  and 
$$
G_j(x_j) = \prod_{i, i\not=j} g_i(x_j) \ge K(m)
$$
where $K(m)$ is a constant depending only on $m$. 
Define 
$$
\phi(x) = x + \sum_{j=1}^n (y_j-x_j)\frac{G_j(x)}{G_j(x_j)}
$$
so that $\phi(x_j) = y_j$ for each $j$. 
Then for $\delta = m\epsilon/n^2$ and for each $j$ : $|y_j-x_j|<\delta$ 
the map $\phi$ is the desired diffeomorphism.
\end{proof}

Fix an admissible function $\rho$ as above and let
\begin{equation}\label{eta}
\eta=\inf_{y\in\S^1} \rho(y)/ \sup_{y\in\S^1} \rho(y)
\end{equation}
Note that $|Dh^k(x)|_\rho\ge \eta Dh^k(x)$ holds for any $x\in\S^1$ and any $k\ge 1$.

We say that a set $U$ is {\em eventually expanding} if there exists a positive integer $k(U)$ such that whenever $k\geq k(U)$ and $x\in U\setminus \Par(h)$, we have $|Dh^k(x)|_\rho>1$. The assertion of Theorem~\ref{thm:circlemap} is that $\S^1$ is eventually expanding. 

\begin{proof}[Completion of proof of Theorem~\ref{thm:circlemap}]
By compactness, it suffices to show that
each $x_0 \in \S^1$ has an eventually expanding neighborhood $U(x_0)$.
% and an integer $k_0=k_0(x_0)$ such that,
%for all $x \in U(x_0) \setminus \Par(h)$ and $k \geq k_0(x_0)$, $$|Dh^k(x_0)|\pp>1.$$

{\em Case 1.} Assume $h^k(x_0)\not\in \Par(h)$ for each $k\ge 0$. Then by Proposition~\ref{prop:expanding} (c),
$Dh^k(x_0)\rightarrow \infty $ as $k \rightarrow
\infty$, so by continuity, there exists a $k_0$ and a neighborhood $U(x_0)$ of $x_0$ such that, for $x \in U(x_0)$,
$Dh^{k_0}(x)\geq \frac{2}{K_0 \eta}.$ By Proposition~\ref{prop:expanding} (b), for all $k \geq k_0$ and $x \in U(x_0)$,
$$ Dh^k(x)=Dh^{k_0}(x) Dh^{k-k_0}(h^{k_0}(x))\ge K_0 Dh^{k_0}(x) \geq \frac{2}{\eta}, $$ hence
$$|Dh^k(x)|\pp \geq \eta Dh^k(x) \geq 2.$$
Thus $U(x_0)$ is eventually expanding.

{\em Case 2.} Assume that $h^k(x_0)\in \Par(h)$ for some $k\ge 0$. By (A2), it suffices to consider the case $x\in \Par(h_0)$.
Reducing $\delta_0$ in (A1) if necessary, we may assume that
$Dh(x)>1$ holds on $B(x_0,\delta_0)\setminus \{x_0\}$.
Let $K=2/\eta$ and let $\delta=\delta(x_0,\delta_0, K)>0$ be a small constant given by Proposition~\ref{prop:expanding} (d).
% (d) of $x_0$, $\delta_0$ and $K=2/\eta$.
%Let $B_n$ denote the component of $h^{-n}(B_0)$ which contains $x_0$ for each $n$. Let $\tau>0$ be the minimal length of the components of $B_0\setminus B_1$ %
Let us prove that $|Dh^k(x)|_\rho>1$ holds for all $x\in B(x_0,\delta)\setminus \{x_0\}$ and $k\ge 1$, so in particular, $B(x_0,\delta)$ is eventually expanding. Indeed, if $x, h(x), \ldots, h^k(x)\in B(x_0,\delta_0)$, then $\rho(h^k(x))>\rho(x)$ and $Dh^k(x)>1$, hence $|Dh^k(x)|_\rho>1$.  Otherwise, we have
$Dh^k(x)> 2/\eta$ which implies that $|Dh^k(x)|_\rho\ge \eta Dh^k(x)\ge 2.$
%Assume $k\ge n$.  Let $J$ be the component of $B_n\setminus B_{n+1}$ which contains $x$. Then the intervals $J, h(J), \ldots, h^{n}(J)$ are pairwise disjoint, %and $h^n(J)$ is a component of $B_0\setminus B_1$. By Lemma~\ref{dist},
%$$Dh^n(x)\ge e^{-C_0}\frac{|h^n(J)|}{|J|}\ge e^{-C_0}\frac{\tau}{|B_n|}> (K_0\eta)^{-1}.$$
%By (b), $Dh^k(x)>\eta^{-1}$. Thus $|Dh^k(x)|_\rho \ge \eta Dh^k(x) >1$.
\end{proof}

\subsection{Geometric expanding properties of topological expanding maps}
This section is devoted to the proof of Proposition~\ref{prop:expanding}. Throughout, fix $h\in \mathcal{T}_d^*$. We shall first establish lower bounds on the derivative of first return maps to small nice intervals.

Recall an open interval $A \subset \S^1$ is \textit{nice} if $h^n(\partial A) \cap A= \emptyset $ for all $n \geq 0.$
Let $$D(A)=\{x\in \S^1: \exists k\ge 1 \text{ such that } h^k(x)\in A\}.$$ 
For each $x\in D(A)$, the {\em first entry time} $k(x)$ is the minimal positive integer such that $h^{k(x)}(x)\in A$. 
The {\em first entry map} $R_A: D(A)\to A$ is defined by $x\mapsto h^{k(x)}(x)$. For $x\in D(A)\cap A$, the entry time is also called the {\em first return time} and the map $R_A|_{D(A)\cap A}$ is called the {\em first return map}.
For a nice interval $A$ and any component $J$ of $D(A)$, 
the entry time $k(x)$ is independent of $x\in J$, 
and if we denote the common entry time by $k$, 
then the intervals $J, h(J),\ldots, h^{k-1}(J)$ are pairwise disjoint 
and $h^k: J\to A$ is a diffeomorphsim. 

There is an arbitrarily small nice interval around any point $z_0\in\S^1$. 
Indeed, let $O$ be an arbitrary periodic orbit such that $h^k(z_0)\not\in O$ for all $k\ge 0$. Then for any $n$, any component of $\S^1\setminus h^{-n}(O)$ is a nice interval.  By~\cite{MMS}, $h$ has no wandering interval which implies that $h^{-n}(O)$ is dense in $\S^1$. The statement follows.

\begin{lemma}\label{2}
For any periodic point $p$ and any constant $K>0$, there exists an arbitrarily small nice interval $A\ni p$ with the following property. Denote by $A'$ the component of $D(A)$ which contains $p$. Then $$DR_A(x)>1\text{ for all }x\in A'\setminus \{p\}$$ and
$$DR_A(x)\ge K\text{ for all }x\in D(A)\cap(A\setminus A').$$
\end{lemma}

\begin{proof}
Let $s_0$ be the period of $p$. Let $B_0\ni p$ be an arbitrary nice interval such that $B\cap\text{orb}(p)=\{p\}.$
For each $n\ge 1$, define inductively $B_n$ to be the component of $h^{-s_0}(B_{n-1})$ which contains $p$. Then $B_n$ is a nice interval for each $n$ and $|B_n|\to 0$ as $n\to\infty$.  Let $$\varepsilon_n=\sup \{|J|: J\text{ is a component of } h^{-i}(B_n)\text{ for some }i\ge 0\}.$$ Since $h$ has no wandering intervals, $\varepsilon_n\to 0$ as $n\to\infty$.

Let $\delta_0$ be the minimum of the length of the components of $B_0\setminus B_1$. Choose $n$ large enough such that
\begin{itemize}
\item $\varepsilon_n\le e^{-2C_0}\delta_0/K$; 
(where $C_0$ is the total variation of $\log Dh$.)
\item $Dh^{s_0}>1$ on $B_{n+1}\setminus \{p\}$ 
(according to the third property defining $\mathcal{T}_d^*$)
.
\end{itemize}

Let us verify that $A:=B_n$ satisfies the desired properties. So let $x\in A\setminus A'=B_n\setminus B_{n+1}$ and let $k\ge 1$ be the first return time of $x$ into $A$. We need to prove that $Dh^k(x)\geq K.$

To this end, let $T$ be the component of $B_n\setminus B_{n+1}$ which contains $x$ and let $J$ be the component of $h^{-k}(B_n)$ which contains $x$. Then $J\subset T$ and $k>ns_0$. Note that $h^{js_0}(T)$ is a component of $B_{n-j}\setminus B_{n-j+1}$ for each $0\le j\le n$. Since the first return time of $p$ to $B_0$ is equal to $s_0$, the intervals $B_1, h(B_1), \ldots, h^{s_0-1}(B_1)$ are pairwise disjoint. Therefore, the intervals $h^j(T)$, $0\le j<ns_0$, are pairwise disjoint. By Lemma~\ref{dist}, $h^{s_0 n} |T$ has distortion bounded by $C_0$. Since $h^{ns_0}(J)$ is a component of $h^{-k+s_0n}(B_n)$, we have $|h^{ns_0}(J)|\le \varepsilon_n$. Therefore,
$$\frac{|J|}{|T|} \le e^{C_0}\varepsilon_n/\delta_0.$$ Since $J, h(J), \ldots, h^{k-1}(J)$ are pairwise disjoint, by Lemma~\ref{dist} again, we obtain
$$Dh^k(x)\ge e^{-C_0} \frac{|B_n|}{|J|}\ge e^{-C_0}\frac{|T|}{|J|}\ge K.$$
\end{proof}

\begin{lemma}\label{highper}
For any $K\geq 1$, there exists $s_0$ such that 
if $p$ is a periodic point with period $s\ge s_0$ then $Dh^s(p)\ge K$. 
In particular, $\text{Par}(h)$ is finite.
\end{lemma}
\begin{proof}
Let $p_0$ be an arbitrary fixed point of $h$ and for each $n=1,2,\ldots$, let
$$\varepsilon_n=\min\{|J|: J\text{ is a component of } \S^1\setminus h^{-n}(p_0)\}.$$
Then $\varepsilon_n\to 0$ as $n\to\infty$.

By Lemma~\ref{2}, there is a small nice interval $A\ni p_0$ such that $DR_A\ge 1$ holds on $A'$ and $DR_A>K\geq 1$ holds on $D(A)\cap(A\setminus A')$, where $A'$ is the component of $h^{-1}(A)$ which contains $p_0$.
Let $\delta$ be the minimum of the length of the components of $A\setminus \{p_0\}$ and let $s_0\ge 2$ be so large that $\varepsilon_s\le \delta/(e^{C_0}K)$ for all $s\ge s_0$.

Now let $p$ be a periodic point with period $s\ge s_0$. We shall prove that $Dh^{s}(p)\ge K$.
Assume first that there exists $p'\in \text{orb}(p)\cap A$. Let $0=t_0<t_1<t_2<\cdots<t_n=s$ the consecutive returns of $p'$ into $A$.  Note that there exists $0\le i_0<n$ such that $h^{t_{i_0}}(p')\in A\setminus A'$,
so $$Dh^s(p)=Dh^s(p')=\prod_{i=0}^{n-1} DR_A(h^{t_i}(p'))\ge K.$$
Now assume that $\text{orb}(p)\cap A=\emptyset$. Let $I$ be an open interval bounded by $p_0$ and some point $p'$ in $\text{orb}(p)$ with the property that $I\cap \text{orb}(p)=\emptyset$. Then $I$ is a nice interval and $|I|\ge \delta$. Let $J$ be a component of $h^{-s}(I)$ which has $p'$ as a boundary point. Then $h^j(J)\cap I=\emptyset$ for $j=1,2,\ldots, s-1$ and $|J| \le \varepsilon_s$. By Lemma~\ref{dist}, we have
$$Dh^s(p)=Dh^s(p')\ge e^{-C_0} \frac{|I|}{|J|}\ge e^{-C_0} \delta /\varepsilon_s\ge K.$$
This proves the first statement. As fixed points of $h^n$ are isolated for each $n\ge 1$, it follows that $\Par(h)$ is finite.
\end{proof}

\begin{lemma}\label{2'}
For each $h\in\mathcal{T}_d^*$, there exists a constant $\lambda_0>1$ such that
for any $x\in\S^1\setminus \Par(h)$, if $A$ is a sufficiently small nice interval containing $x$, then $DR_A\ge \lambda_0$ holds on $D(A)\cap A$.
%$y\in A$ and $k\ge 1$ is the first return time of $y$ into $A$, then $Dh^k(y)\ge \lambda_0$.
\end{lemma}
\begin{proof} By Lemma~\ref{highper}, there exists $s_0$ such that if $p$ is a periodic point with period $s>s_0$ then
$Dh^s(p)\ge 2 e^{C_0}$. Let $1<\lambda_0<\lambda_1<2$ be a constant such that if $p\not\in \Par(h)$ is a periodic point of period $s\le s_0$, then $Dh^s(p)>\lambda_1$. Let $\delta>0$ be a small constant such that $|Dh^s(x_1)-Dh^s(x_2)|<\lambda_1-\lambda_0$ 
whenever $s\le s_0$ and $\text{dist}(x_1,x_2)<\delta$.

Now let $x\in \S^1\setminus \Par(h)$ and let $A\ni x$ be a nice interval such that $|A|<\delta$ and $\overline{A}\cap\Par(h)=\emptyset$. Now consider $y\in A$ with $k\ge 1$ as the first return time of $y$ to $A$. Let $J$ be the component of $h^{-k}(A)$ which contains $y$. Then $h^k: J\to A$ is a diffeomorphism with distortion bounded by $C_0$. Since $J\subset A$, there is a fixed point $p$ of $h^k$ in $\overline{J}$. Note $p\not\in \Par(h)$. Since $h^j|J$ is monotone increasing for all $0\le j\le k$, $k$ is equal to the period of $p$. If $k\le n_0$ then $Df^k(p)\ge \lambda_1$ and since $|J|\le |A|<\delta$, we have $Df^k(y)\ge Df^k(p)-(\lambda_1-\lambda_0)\ge \lambda_0$. If $k>n_0$, then $Dh^k(p)\ge 2e^{C_0}$, and hence
$Dh^k(y)\ge e^{-C_0}Dh^k(p)\ge 2>\lambda_0.$
\end{proof}

\iffalse
\begin{lemma}\label{3}
For each $h\in\mathcal{T}_d^*$, there exists a constant $K_0>0$ such that, for each $x \in \S^1$ and $k \geq1$, we have $Dh^k(x)\geq K_0. $
\end{lemma}
\begin{proof}

\end{proof}

%Let $\delta_0>0$ be a small constant such that the balls $B(p,2\delta_0)$, $p\in\Par$, are pairwise disjoint and such that
%$Df(x)>1$ holds for $x\in B(p,2\delta_0)\setminus \{p\}$ for each $p\in\Par$. For any $x\in\S^1$, let
%$$e(x)=\inf\{n\ge 0: f^n(x)\not\in \bigcup_{p\in\Par}B(p,\delta_0)\}.$$

\begin{lemma}\label{4}
Let $h\in\mathcal{T}_d^*$. Then for each $x \in \S$, one of the following holds:
\begin{enumerate}
 \item $h^k(x)\in\Par(f)$ for some $k \geq0$.
 \item $Dh^k(x)\rightarrow \infty $ as $k \rightarrow \infty$.
\end{enumerate}
\end{lemma}
\begin{proof}

\end{proof}
\fi

\begin{proof}[Proof of Proposition~\ref{prop:expanding}]
(a). This property was proved in Lemmas~\ref{highper}.

(b). By Lemmas~\ref{2} and ~\ref{2'}, for any $y\in \S^1$ there is a nice interval $A(y)\ni y$ such that the derivative of the first return map is at least $1$. By compactness, there exist $y_1, y_2, \ldots, y_r\in\S^1$ such that $\bigcup_{i=1}^r A(y_i)=\S^1$. Now consider an arbitrary $x\in\S^1$ and $k\ge 1$. Define a sequence $\{i_n\}\subset \{1,2,\ldots, r\}$ and $\{k_n\}$ as follows. First let $k_0=-1$, take $i_0$ such that $x\in A(y_{i_0})$ and let $k_1=\max\{1\le j\le k: h^j(x)\in A(y_{i_0})\}$. If $k_1=k$ then we stop. Otherwise, take $i_1\subset \{1,2,\ldots, r\}\setminus\{i_0\}$ be such that $h^{k_1+1}(x)\in A(y_{i_1})$ and let $k_2=\max\{k_1<j\le k: h^j(x)\in A(y_{i_1})\}$. Repeat the argument until we get $k_n=k$. Then $n\le r$ and $Dh^{k_{j+1}-k_j-1}(h^{k_j+1}(x))\ge 1$. It follows that $$Dh^k(x)\ge \prod_{i=1}^{n-1} Dh(h^{k_i}(x))\ge \left(\min_{y\in\S^1} Dh(y)\right)^{r-1}.$$ This proves the property (b).

(c). Assuming $h^k(x) \not\in \Par(h)$ for all $k \geq0$, let us prove that $Dh^k(x)\rightarrow \infty $ as $k \rightarrow \infty$. By (b), it suffices to show that $\limsup_{k\to\infty} Dh^k(x)=\infty$. Let $y\in\omega(x)\setminus \Par(h)$ and consider a small nice interval $A$ containing $y$ for which the conclusion of Lemma~\ref{2'} holds. Since $y\in\omega(x)$ there exist $n_1<n_2<\cdots$ such that $h^{n_k}(x)\in A$. By Lemma~\ref{2'}, $Dh^{n_{k+1}-n_k}(h^{n_k}(x))\ge \lambda_0>1$ for all $k$. Thus $Dh^{n_{k+1}}(x)\ge Dh^{n_1}(x) \lambda_0^{k}\to\infty$ as $k\to\infty$.

(d). The proof repeats part of the proof of Lemma~\ref{2}. Let $B_0$ be a nice interval such that $B_0\subset B(p,\delta_0)$, $B_0\cap \text{orb}(p)=\{p\}$. Define $B_n$ to be the component of $h^{-n}(B_0)$ which contains $p$. Let $\tau>0$ be the minimal length of the components of $B_0\setminus B_1$. Given $K>0$ let $n_0$ be so large that $|B_{n_0}|< e^{-C_0}\tau K_0/K$. Choose $\delta>0$ such that $B(p,\delta)\subset B_{n_0}$.

Now assuming that $d(x,p)<\delta$ and $\max_{j=1}^{k} d(h^{j}(x), p)\ge \delta_0$, let us prove $Dh^{k}(x)\ge K.$
Let $n\ge n_0$ be such that $x\in B_n\setminus B_{n+1}$. Note that $k> n$. Let $J$ be the component of $B_n\setminus B_{n+1}$ which contains $x$, then the intervals $J, h(J), \ldots, h^{(n-1)}(J)$ are pairwise disjoint, $h^{n}(J)$ is a component of $B_0\setminus B_1$.  Thus by Lemma~\ref{dist},
$$Dh^{n}(x)\ge e^{-C_0} \frac{|h^n(J)|}{|J|}\ge e^{-C_0}\frac{\tau}{|B_{n_0}|}\ge K/K_0.$$
By (b), it follows that $Dh^{k}(x)\ge K_0 Dh^{n}(x)\ge K.$

(e). Without loss of generality, we may assume that all periodic points in $\Par(h)$ are fixed points.
Let $X_0=\Par(h)$ and for $n\ge 1$, let $X_n=h^{-n}(\Par(h))\setminus h^{-n+1}(\Par(h))$. So for each $y\in X_n$, $n$ is the minimal integer such that $h^n(y)\in \Par(h)$.

Let $\delta_0>0$ be a small constant such that $h|_{B(p,\delta_0)}$ is injective and 
$B(p, \delta_0)\cap \Par(h)=\{p\}$ for each $p\in \Par(h)$. Note that this choice of $\delta_0$ implies the following: if $y\in B(p,\delta_0)\cap X_m$ for some $m\ge 1$, then $\max_{j=1}^m d(h^j(y), p)\ge \delta_0$. Thus by (d), there is a constant $\delta>0$ with the following property: if $y\in B(p,\delta)\cap X_m$ for some $m\ge 1$, then $Dh^m(y)\ge K/K_0.$

Now for each $p\in\Par(h)$, fix a nice interval $A_p\ni p$ such that $A_p\subset B(p,\delta)$.
Given $x\in X_n$ with $n\ge 1$, we shall estimate $Dh^n(x)$ from below. Let $p=f^n(x)$.

{\em Case 1.} 
Assume that there exists $0\le j<n$ such that $y:=h^j(x)\in B(p,\delta)$. Then $y\in X_{n-j}\cap B(p,\delta)$ and hence $Dh^{n-j}(x)\ge K/K_0$.
By (b), it follows that $Dh^n(x)\ge K_0 Dh^{n-j}(y)\ge K.$

{\em Case 2.} Assume now that $h^j(x)\not\in B(p,\delta)$ for all $0\le j<n$. Then $n$ is the first entry time of $x$ into $A_p$. Let $J$ be the component of $h^{-n}(A_p)$ which contains $x$. Then $J, h(J),\ldots, h^{n-1}(J)$ are pairwise disjoint. By Lemma~\ref{dist}, $Dh^n(x)\ge e^{-C_0} |A_p|/|J|$. Provided that $n$ is large enough, $|J|$ is small so that $Dh^n(x)\ge K$.
\end{proof}

\section{Parabolic external maps}\label{parabolicexternalmaps}
% THINK ABOUT STRUCTURE AND INTRO. WE WANT TO SAY
% \begin{enumerate}
%  \item In this section we will prove Theorem \ref{th}, which relates parabolic external maps to topologically expanding maps and to metrically expanding maps,
% and which completes the theory of parabolic-like maps.
% \item what is a parabolic-like map: fast intro
% \item why do we need all of this: fast intro to external maps in the 2 cases
% \end{enumerate}
In this section we will prove Theorem \ref{th}, which relates parabolic external maps to topologically expanding maps and to metrically expanding maps,
and which completes the theory of parabolic-like maps.
We will start by giving an introduction to parabolic-like maps. We will always assume the degree $d \geq 2$, if not specified otherwise. 
\subsection{Parabolic-like maps}
The notion of parabolic-like maps is modeled on the notion of polynomial-like maps 
and can be thought of as an extension of the later theory.
A polynomial-like map is an object which encodes the dynamics of a polynomial on a neighborhood of its filled Julia set. We recall 
that the filled Julia set for a polynomial is the complement of the basin of attraction of the superattracting fixed point $\infty$, and 
therefore the dynamics of a polynomial is expanding on a neighborhood of its filled Julia set.

A (degree $d$) polynomial-like mapping is a (degree $d$) 
proper holomorphic map $f: U' \rightarrow U$, where $U', U\approx\D$ and $\overline {U'} \subset U$.
The filled Julia set for a polynomial-like map $(f,U',U)$ is the set of points which never leave $U'$ under iteration.
Any polynomial-like map is associated with an external map, which encodes the dynamics of the polynomial-like map outside of its
filled Julia set, so that a polynomial-like map is determined (up to holomorphic conjugacy) 
by its internal and external classes together with their matching number in $\Z/(d-1)\Z$. 
By replacing the external map of a degree $d$ polynomial-like map 
with the map $z \rightarrow z^d$ (which is an external map of a degree $d$ polynomial) via surgery, 
Douady and Hubbard proved that any degree $d$ polynomial-like map can be straightend 
(this is, hybrid conjugate) to a degree $d$ polynomial (see \cite{DH}).\\

On the other hand, in degree $2$ a parabolic-like map is an object encoding the dynamics 
of a member of the family $P_A(z)=z+1/z +A\in Per_1(1)$, where $A \in \C$, on 
a neighborhood of its filled Julia set $K_A$. This family can be characterized 
as the quadratic rational maps with a
parabolic fixed point of multiplier $1$ at $\infty$, and critical points at $\pm 1$. 
The filled Julia set $K_A$ of $P_A$
is the complement of the parabolic basin of attraction of $\infty$. So on a neighborhood of the filled Julia set $K_A$ of a map $P_A$ there exist an
attracting and a repelling direction.

A degree $d$ parabolic-like map is a 4-tuple ($f,U',U,\gamma$) where $U', U,U \cup U',\approx\D$, $U'\nsubseteq U$, 
$f:U' \rightarrow U$ is a degree $d$ proper holomorphic map with a parabolic fixed point at $z=z_0$ of
 multiplier 1, and with a forward invariant arc $\gamma:[-1,1] \rightarrow \overline {U}$, which we call dividing arc, emanating from $z_0$
 such that:
 \begin{itemize}
 \item $\gamma$ is $C^1$ on $[-1,0]$ and on $[0,1]$, and $\gamma(\pm 1) \in \partial U$,
  \item $f(\gamma(t))=\gamma(dt),\,\,\, \forall -\frac{1}{d} \leq t \leq \frac{1}{d},$ and $\gamma([ \frac{1}{d}, 1)\cup (-1, -\frac{1}{d}]) \subseteq U \setminus U',$
  \item it divides $U',U$ into $\Omega', \Delta'$ and $\Omega, \Delta$
respectively, such that $\Omega' \subset \subset U$ 
(and $\Omega' \subset \Omega$) and $f:\Delta' \rightarrow \Delta$ is an isomorphism.
 \end{itemize}
 The filled Julia set is defined in the parabolic-like case to be the set of points which do not escape $\Omega' \cup \gamma$
under iteration. As for polynomial-like maps, any parabolic-like map is associated with an external map (see \cite{L}), so that a parabolic-like map is determined (up to holomorphic conjugacy) 
by its internal and external classes. 
By replacing the external map of a degree $2$ parabolic-like map with the map
$h_2(z)=\frac{z^2+1/3}{z^2/3+1}$, (which is an external map of any member of the family $Per_1(1)(z)=\{[P_A]| P_A(z)= z+ 1/z +A\}$, as shown in \cite{L})
one can prove that any degree $2$ parabolic-like map is hybrid equivalent to a 
member of the family $Per_1(1)$ (see \cite{L}).

The notion of parabolic-like map can be generalized to objects with a finite number of parabolic cycles.
More precisely, let us call
\textit{simply parabolic-like maps} the objects defined before, which have a unique parabolic fixed point. Then a
\textit{parabolic-like map} is a 4-tuple ($f,U',U,\gamma$) where $U', U,U \cup U',\approx\D$, $U'\nsubseteq U$, $f:U' \rightarrow U$ is a degree $d$ proper holomorphic map with 
a finite set $Par(f)$ of parabolic points $p$
of multiplier $1$, such that for all $p \in Par(h)$ there exists a dividing arc $\gamma_p \subset \overline U$,  $p\in \gamma_p$,
smooth except at $p$, $\gamma= \bigcup_{p} \gamma_p$, and such that:
\begin{itemize}
 \item for $p \neq p'$, $\gamma_p \cap \gamma_{p'} = \emptyset$ and $f:\gamma_p \cap U' \rightarrow \gamma_{f(p)}$ is a diffeomorphism,
 \item it divides $U$ and $U'$ in $\Omega_p,\,\Delta_p$ and $\Omega'_p,\,\,\Delta'_p$ respectively, all connected, and such that 
 $f: \Delta'_p \rightarrow \Delta_{f(p)}$ is an isomorphism and $\Omega'_p \subset \Omega_p$,
 \item calling $\Omega= \bigcap_p \Omega_p$ and $\Omega'= \bigcap_p \Omega'_p$, we have $\Omega' \subset \subset U$.
\end{itemize}
The filled Julia set for a parabolic-like map ($f,U',U,\gamma$) is (again) the set of points that never leave $\Omega' \cup \gamma$
under iteration.
\subsection{External maps for parabolic-like maps}\label{extmap}
The construction of an external map for a simply parabolic-like map $(f,U',U, \gamma)$ with connected filled Julia set $K_f$ is relatively easy, and it shows that this map belongs to $\P_d$.
%In the case $K_f$ is connected, we can construct easily an external map for a simply parabolic-like map $(f,U',U, \gamma)$.
Indeed consider the Riemann map 
$\alpha: \widehat \C \setminus K_f \rightarrow \widehat \C \setminus \D$, normalized by fixing infinity and by setting $\alpha(\gamma(t))\rightarrow 1$ as $t \rightarrow 0$.
Setting $W_+= \alpha (U \setminus K_f)$ and
$W_+'= \alpha (U' \setminus K_f)$, we can define a degree $d$ covering $h^+:= \alpha \circ f \circ \alpha^{-1}: W_+' \rightarrow W_+$,
reflect the sets and the map with respect to the unit circle, and the restriction to the unit circle 
$h: \S^1 \rightarrow \S^1$ is an external map for $f$. An external map for a parabolic-like map is defined up to real-analytic diffeomorphism.
From the construction it is clear that $h \in \P_{d,1}$. 
The construction of an external map for a simply parabolic-like map with disconnected filled Julia set is 
more elaborated (see \cite{L}), and still produces a map in $\P_{d,1}$.
Repeating the costructions handled in \cite{L} for (generalized) parabolic-like maps, one can see that the external map for a degree $d$ parabolic-like map belongs to $\P_d$.

On the other hand, it comes from the Straightening Theorem for parabolic-like mappings (see \cite{L})
that a map in $\P_{d,1}$ is the external map for a parabolic-like map
(with a unique parabolic fixed point) of same degree (and the
proof is analogous in case of several parabolics fixed points and parabolic cycles).

% As we said before the Straightening theorem is proven, for polynomial-like maps, by replacing the external map of a degree $d$ polynomial-like map 
% with the map $z \rightarrow z^d$, which is an external map for any degree $d$ polynomial. The space 
% of external classes of polynomial-like mappings is easily characterized 
% as those circle coverings which are q-s.-conjugate to $z \mapsto z^d$ for some $d\geq 2$.
% Here we will prove the same statement for degree $2$ simply parabolic-like maps, and bla.
% The external map of a degree $d$ polynomial-like map is quasisconformally conjugate 
% to the map $z \rightarrow z^d$, 
% which is the external class of any degree $d$ polynomial. 
While the space of external classes of polynomial-like mappings is easily characterized 
as those circle coverings which are q-s.-conjugate to $z \mapsto z^d$ for some $d\geq 2$,
this is not the case for parabolic external classes. Theorem \ref{th} gives a characterization for these maps.
%This motivates the following discussion.

\subsection{Proof of Theorem~\ref{th}}\label{proof}
The main technical difficulty for proving Theorem \ref{th} is to prove the following property for maps in $\M_d$:
\REFLEM{almostsamepowerseries}
For any $h\in\M_d$ there is a map $\phi\in\F_1$ 
such that the map $\chh:= \phi\circ h\circ \phi^{-1}$ also belongs to $\M_d$ 
and in addition for every orbit $\chp_0, \chp_1, \dots \chp_s=\chp_0\in\Par(\chH)$ 
say of parabolic multiplicity $2n$, 
the power series developments of $\chH : \TT \to \TT$ at the points $\chp_k$, $k\in\Z/s\Z$, take the form 
\REFEQN{uniformH}
\chH(x) = \chp_{k+1} + (x-\chp_k)(1+(x-\chp_k)^{2n}\cdot 
\chP(x-\chp_k) + \OO(x-\chp_j)^{6n})
\ENDEQN
for some fixed polynomial $\chP$ (i.e.~$\chP$ depends on the cycle, but not on $k$) 
with non-zero constant term and degree at most $4n-1$. 
\ENDLEM

We will first prove the Theorem assuming the Lemma. We prove Lemma \ref{almostsamepowerseries} in Section \ref{proofoflemma}. The following Proposition proves that $\M_d\subset\P_d$.

\REFPROP{MdinPd}
For every $h\in\M_d$ there exists $\eps_0> 0$ 
such that for every $0<\eps\le\eps_0$ 
the map $h$ has a holomorphic extension $(h, W', W, \ga)$ 
as a parabolic external map with range $W\subseteq \{z: |\log|z|| < \eps\}$. 
In particular $\M_d\subset\P_d$ and any map 
which is conjugate to $h$ by $\phi\in\F_1$ also belongs to $\P_d$.
\ENDPROP
\PROOF
It suffices to consider maps $h\in\M_d$ satisfying the 
properties of $\chh$ in the Lemma above. 
Also it suffices to work with the representative $H : \TT \to \TT$ of $h$.
Since $H$ is a real analytic covering map 
it extends to a holomorphic isomorphism
$H : V' \to V$ between reflection symmetric neighborhoods of $\R$ 
and satisfying $H(z+1) = H(z) + d$. 
For each $p\in P = E^{-1}(\Par(h))$, choose a pair of repelling Fatou cooordinates 
$\phi_{p}^\pm : \Xi_{p}^\pm \to \H_l :=\{z | \Re(z) < 0 \}$ such that 
each $\phi_p^\pm$ and each $\Xi_{p}^\pm$ 
is symmetric with respect to $\R$, and $\phi_{p+1}^\pm = \phi_p^\pm(x-1)$. 
Possibly restricting the $\phi_p^\pm$ we can suppose all the domains 
$\Xi_{p}^\pm$, with $p$ ranging over $P$, disjoint for each choice of sign,  
$H(\Xi_{p}^\pm) \supset \Xi_{h(p)}^\pm$ and $H^s$ is univalent on 
$\Xi_{p}^\pm$, where $s=s_p$ denotes the period of $E(p)$.

For each orbit in $\Par(h)$ choose a representative $p_i\in P$, 
and call $2n_i$ the parabolic multiplicity of the orbit.
Define $S_{\eps} := \{x+iy | |y| < \eps\}$ for $\eps>0$.  
For $p$ in the orbit of $p_i$, call $C_p$ the double cone,
symmetric with respect to the real line, 
such that $C_p \cap \R =\{p\}$ and the angle 
between $\R$ and $\partial C_p$ is $1/(16 n_i)$. 
Call $X_{\epsilon}= S_{\eps} \setminus \bigcup_{p \in P} C_p$.
By a compactness argument, since $DH(z)>1$ for $z \in \R \cap X_{\eps}$, 
and $\limsup_{z \in \partial X_{\eps}, z\rightarrow p}|Arg(DH(z)-1)| = \pi/4$,
there exists an $\epsilon_0 >0$ such that, for all $z \in X_{\epsilon_0}$, $\Re(DH(z))>1$.
Possibly decreasing $\epsilon_0$, we can assume that for all $p,p' \in P$, $\overline{S_{\epsilon_0}} \cap C_p \cap C_{p'}= \emptyset$.
%%So $h^{-1}(X_{\epsilon_0}) \subset \subset S_{\epsilon_0}$.
Since $h$ satisfies the conclusion of \lemref{almostsamepowerseries} 
the curves $(\phi_p^\pm)^{-1}(\mp a i + \R_-)$ intersects the boundary of $C_p$ 
at angle $\pi/4$ asymptotically as $a\to\infty$ and moreover for $E(p)$ and $E(p')$ in 
the same orbit this happens asymptotically at the same imaginary height. 
Thus possibly decreasing $\eps_0$ and fixing any $\eps, 0 < \eps < \eps_0$ 
we may choose $a_i>0$ (depending on $\eps$) such that for all $i$ and all $p$ with $E(p)$ 
in the orbit of $E(p_i)$ the arcs 
$\gamma_p^\pm = (\phi_p^\pm)^{-1}(\mp a_i\cdot i + \R_-)$ 
exits $X_\eps$ through $\partial S_\eps$ transversally.

Let $\Delta_p^u$ be the closed connected component in $S_\eps$ bounded by 
$\ga_p:=\ga_p^-\cup{p}\cup\ga_p^+$ and containing $C_p$, 
and set $\Delta_p= \Delta_p^u \cup \tau(\Delta_p^u)$. 
Define $\hat X'_\eps= S_\eps \setminus \bigcup_p \Delta_p$ 
(note that $\hat X'_\eps \subset X_\eps$), 
and $\hat X_\eps =  S_\eps \setminus \bigcup_p \Delta_{H(p)}$.

Then by construction $H^{-1}(\hat X_\eps) \subset \hat X'_\eps$ and 
$\overline{H^{-1}(\hat X_\eps)} \subset S_\eps$. 
Thus taking $W := \exp(S_\eps)$, $W' := h^{-1}(W)$ and the multi arc $\ga$ 
as the family $\exp(\ga_p)$, $p\in P$, we have constructed an extension 
$(h, W', W, \ga)$ of $h$ in $\P_d$.
\ENDPROOF

\REFLEM{hdinPd}
The map $h_d$ is {\Mobius} conjugate to a map in $\M_d$, so $h_d\in\P_{d,1}$.
\ENDLEM
\PROOF
For $0<r<1$ define $M_r(z) = (z+r)/(1+rz)$. 
Then $|M_r'(z)|$ is a monotone decreasing function of $\Re(z)$ with 
$|M'_r(-1)| = (1+r)/(1-r)$ and $|M'_r(1)| = (1-r)/(1+r)$. 
Note that for $r= (d-1)/(d+1)$ we have $h_d = M_r(z^d)$. 
Thus $M_r^{-1}\circ h_d \circ M_r = (M_r)^d$ and this map evidently belongs to $\M_{d,1}$.
\ENDPROOF

\noindent\textbf{Proposition~\ref{topgivesqs}:}
\emph{Suppose $h_1, h_2\in\P_d$ are topologically conjugate by an orientation preserving 
homeomorphism $\phi$, which preserves parabolic points. 
Then $\phi$ is quasi-symmetric.} 

\PROOF{}
Let $(h_i, W_i', W_i, \ga^i)$, $i= 1, 2$ be holomorphic extensions 
with $W_i'$ and $W_i$ bounded by $C^1$ Jordan curves intersecting $\ga^i$ transversely.
The case $h_i\in\P_2^1$ is handled in Lomonaco, \cite{L}. 
The general case is completely analogous, we include the details for completeness. 
It suffices to construct a quasi-conformal extension, 
$\phi : \overline{W}_1^+ \to \overline{W}_2^+$, 
with $\phi(\ga^1_p(t)) := \ga_{\phi(p)}^2(t)$ for each $p\in\Par(h_1)$ and 
with $\phi\circ h_1 = h_2\circ\phi$ on $\Omega_1'$.

For each $p\in\Par(h_1)$ extend $\phi$ so that 
$\phi(\ga_p(t)^1) := \ga_{\phi(p)}^2(t)$. 
It is proved in \cite{L} that the arcs $\ga_p^1$ amd $\ga_{\phi(p)}^2$ are quasi-arcs and 
that this extension, which is $C^1$ for $z\not= p$, is quasi-symmetric. 
Next extend $\phi$ as a diffeomorphism 
between the outer boundary of $W_+^1$ and $W_+^2$ respecting the intersections 
with $\ga^i$, 
i.e.~besides being a diffeomorphism it satisfies 
$\phi(\ga_p^1(\pm1))= \ga_{\phi(p)}^2(\pm1)$. 
Then $\phi$ is defined as a quasi-symmetric homeomorphism 
from the quasi-circle boundary of $\Delta_p^1$ 
to the quasi-circle boundary of $\Delta_{\phi(p)}^2$ 
for each $p\in\Par(h_1)$. 
We extend $\phi$ as a quasi-conformal homeomorphism 
$\phi : \Delta_p^1 \to \Delta_{\phi(p)}^2$. 
Next consider the $C^1$ lift 
$\wtphi : \partial W_1'\to\partial W_2'$ 
of $\phi\circ h_1$ to $h_2$ respecting the dividing multi arcs. 
We next extend $\phi$ by $\wtphi$ on $\partial W_1'\cap W_1^+$.
For each $i=1,2$ the connected components of $W_i^+\sm \overline{W}_i'$ 
are quadrilaterals $Q_p^i$ indexed by the $p\in\Par(h_i)$ 
preceding $Q_p^i$ in the counter-clockwise ordering.
Moreover $\phi$ thus defined restricts to a piecewise $C^1$ and hence 
quasi-symmetric homeomorphism from the boundary of $Q_p^1$ to 
the boundary of $Q_{\phi(p)}^2$. 
Extend this boundary homeomorphism to a quasi-conformal homeomorphism 
between $Q_p^1$ and $Q_{\phi(p)}^2$. 

Call the thus extended map $\phi_1$ and its domain and range $U_1^1$ and 
$U_2^1$ respectively. 
Define recursively for $i= 1,2$ and $n\geq 1$:
$$
U_i^{n+1} := U_i^n\cup (h_i^{-1}(U_i^n)\cap\Omega_i')
$$ 
Moreover define recursively $\phi_{n+1}: U_1^{n+1}\to U_2^{n+1}$ 
as the quasi-conformal extension of $\phi_n$, 
which on $h_1^{-1}(U_1^n)\cap\Omega_1'$ satisfies 
$$
\phi_n\circ h_1 = h_2\circ \phi_{n+1}
$$
(i.e.~lift of $\phi_n\circ h_1$ to $h_2$).
Then $\phi_n\cup\phi$ converges uniformly to a quasi-conformal homeomorphism 
$\phi_\infty : \overline{W}^+_1 \to \overline{W}^+_2$, 
which conjugates dynamics except on $\Delta'_1$. 
Thus $\phi$ is the restriction to $\Sen$ of a quasi-conformal homeomorphism 
and thus it is a quasi-symmetric map.
\ENDPROOF
\REFPROP{PdinTdstarcapFd}
$\P_d\subseteq\T_d^*\cap\F_d = \T_d$.
\ENDPROP
\PROOF
%%Define $\whOmplus := \Om\sm\D$ and $\whOmpplus := \Om'\sm\D$, 
%%For $h\in\P_d$ any $p\in\Par(h)$ has a pair of $\tau$-symmetric repelling 
%%petals along the circle and thus $DH^s(x)>1$ on a punctured neighborhood 
%%$I(p)\sm\{p\}$ of $p$.
Let $h\in\P_d$ and let $(h, W', W, \ga)$ be a degree $d$ holomorphic extension of $h$ 
as a parabolic external map with dividing multi arc $\gamma$ 
and associated sets $\Delta'$ and $\Delta$. 
We shall first redefine $\Om$ and $\Om'$ so as to be $\tau$-symmetric :
$\Omega=W\sm\overline{\Delta\cup\tau(\Delta)}$ 
and $\Omega' = W'\sm\overline{\Delta'\cup\tau(\Delta')}$  
then $\Om'':= h^{-1}(\Om)\subset \Om' \subsetneq\Om$. 
It follows that each $p\in\Par(h)$, say of period $n$, 
is admits the circle as repelling directions. 
Indeed if not then it would have a $\tau$-symmetric attracting petal 
along $\Sen$ to one or both sides. 
However since $\Om''\subset \Om$ the parabolic basin for $h^n$ 
containing such a petal would be a proper basin and thus would contain a critical point.

To prove that all other periodic orbits are repelling, 
let $\rho$ denote the hyperbolic metric on $\Om$.
Then each connected component $V$ of $\Omega''$ is a subset of $U\cap W'$ 
for some connected component $U$ of $\Omega$. 
Thus $h$ is expanding with respect to the conformal metric $\rho$. 
Since any non parabolic orbit is contained in $\Omega''\cap\Sen$ 
it follows that all non parabolic orbits are repelling. 
This proves the first inclusion.
The equality sign is immediate from Corollary~\ref{TdstarequalsTdoneplusBV}.
\ENDPROOF
\textit{Completion of proof of Theorem \ref{th}:}\label{proofofth} 
By Proposition~\ref{MdinPd}, $\M_d\subset\P_d$ (and so $\M_{d,1}\subset\P_{d,1}$), and by Proposition~\ref{PdinTdstarcapFd}, $\P_d \subseteq \T_d$. Since $h_d$ is topologically expanding
we have that $\X_{d,1}\subset\T_{d,1}$, and combining Lemma~\ref{hdinPd} and Proposition~\ref{MdinPd} we obtain $\P_{d,1}\subset\X_{d,1}$. So:
$$
\M_d\subset\P_d\subset\T_d\qquad\textrm{and}\qquad\M_{d,1}\subset\P_{d,1}\subset\X_{d,1}\subset\T_{d,1}
$$
By Theorem~\ref{TdisessMd}, any $h\in\T_d$ is real analytically conjugate to a map $\whh\in\M_d$, and so by Proposition~\ref{MdinPd} we also have $h\in\P_d$.
So we obtain 
$$
\P_d=\T_d \qquad\textrm{and}\qquad \pi_d(\M_d)=\pi_d(\P_d) =\pi_d(\T_d).
$$
% the Theorem.
% 
% $\P_d=\T_d$,
% To complete the proof let $h\in\T_d$, then by Theorem~\ref{TdisessMd} 
% the map $h$ is conjugate to a map $\whh\in\M_d$. 
% Thus by Proposition~\ref{MdinPd} we also have $h\in\P_d$. 
% This proves the first displayed formulas in Theorem~\ref{th}. 
% The second and third lines of displayed formulas follow from the first and 
% Theorem~\ref{TdisessMd}.

\subsection{Proof of Lemma \ref{almostsamepowerseries}}\label{proofoflemma}
This section is completely devoted to proving Lemma \ref{almostsamepowerseries}.
Let us start by noticing that it follows from the definition of $\M_d$ 
that $h$ only has finitely many parabolic points. 
The proof of Lemma \ref{almostsamepowerseries} uses the idea of the proof of \corref{mainBV} to recursively 
construct conjugacies to maps 
which full-fills the requirements of $\chH$ to higher and higher orders. 
It turns out that after two steps of the recursion 
we arrive at the desired map $\chH$ and obtain the conjugacy as the composition 
of the pair of conjugacies from the recursion.

The recursion is given by the following procedure: \\
Let $h\in\M_d$ be arbitrary, let $N = N_h$ denote the least common multiple 
of the periods of parabolic orbits for $h$ 
and let $L:= (d^N-1)/(d-1)$. 
Define a real analytic diffeomorphism {\mapfromto {\phi} \R \R} and 
a new real analytic diffeomorphism $\whH$ (lift of degree $d$ covering $\whh$) 
as follows: 
$$
\phi(x) := \frac{1}{L}\sum_{k=0}^{N-1} (H^k)(x)
\qquad\textrm{and}\qquad
\whH := \phi\circ H \circ \phi^{-1}.
$$
Then $\phi(x+1) = 1 + \phi(x)$,
$\Par(\whH) = \phi(\Par(H))$, $N_\whh = N_h$ and
\begin{align}\label{whHinMd}
\whH'(\phi(x)) &= 
\frac{\phi'(H(x))\cdot H'(x)}{\phi'(x)} 
= \frac{H'(x)\sum_{k=0}^{N-1} (H^k)'(H(x))}{\sum_{k=0}^{N-1} (H^k)'(x)}\nonumber\\
&= \frac{\sum_{k=1}^{N} (H^k)'(x)}{\sum_{k=0}^{N-1} (H^k)'(x)} 
= \frac{(H^N)'(x) + \sum_{k=1}^{N-1} (H^k)'(x)}{1+\sum_{k=1}^{N-1} (H^k)'(x)} 
\geq 1
\end{align}
%Thus if $H\in\M_d$ then $H'(\phi(x))\geq 1$ 
with equality if and only if $x\in\Par(H)$, thus $\whH\in\M_d$.

For $p\in\Par(H)$ with period $s$, 
set $\whp := \phi(p)\in\Par(\whH)$, 
$p_k := H^k(p)$, $\whp_k = \phi(p_k) = \whH^k(\whp)$, then $p_{s+k}-p_k = \whp_{s+k}-\whp_k\in\Z$ for each $k \geq 0$.
Let $2n>0$ denote the common parabolic degeneracy. 
%%Then since $h\in\M_d$ we have for $x$ near $p=p_0$ 
%%$$
%%H^s(x) = p_s + (x-p)(1+(x-p)^{2n}\cdot P(x-p) + \OO(x-p)^{4n})
%%$$
%%for some polynomial $P$ of degree at most $2n-1$ and with positive constant term.
A priori the power series developments (Taylor expansions) 
of $H$ around the points $p_k$ 
could have non-linear terms of order less than $2n+1$. 
However, since $h\in\M_d$ the leading non-linear term must be of odd order, say $2m+1$ (and 
have positive coefficient), and Claim \ref{goodformula} implies $m=n$.

Write $h_0:= h$, $H_0 := H$ and $\phi_0:= \phi$. Set $H_1 := \whH$, and define 
$$
\phi_1(x) := \frac{1}{L}\sum_{k=0}^{N-1} (H_1^k)(x)
\qquad\textrm{and}\qquad
H_2 := \phi_1\circ H_1 \circ \phi_1^{-1}.
$$
%Finally, apply Claim \ref{doublygoodformula} to $H_1$.
Then $\phi := \phi_1\circ\phi_0$ and 
$\chH := H_2 $ satisfy the Lemma, with $\chP := \whP+x^{2n}\whR$, where $\whP$ and $\whR$ are given by Claim \ref{doublygoodformula} applied to $H_1$.
\subsubsection{Claim \ref{goodformula} and Claim \ref{doublygoodformula}}
\REFCLAIM{goodformula}
Suppose that for some $m>0$ the Taylor expansions of $H$ around 
the points $p_k$ take the form
$$
H(x) = p_{k+1} + (x - p_k)(1 + (x - p_k)^{2m}\cdot P_k(x - p_k) + \OO(x - p_k)^{4m}),
$$ 
where $P_k$ is a polynomial of degree at most $(2m-1)$, 
$P_{s+k}=P_k$ for $k\geq 0$ 
and where $P_k(0) > 0$ for at least one $k, 0\leq k < s$. 
Then for each $k$ the Taylor approximation to order $4m$ of $\whH$ 
at $\whp_k$ takes the form
\REFEQN{orderfourm}
\whH(\whx) = \whp_{k+1} + (\whx - \whp_k)
(1 + (\whx - \whp_k)^{2m}\cdot \whP(\whx - \whp_k) + \OO(\whx - \whp_k)^{4m}),
\ENDEQN
where
\REFEQN{sumpoly}
\whP(x) := \frac{L^{2m}}{s}  \cdot \sum_{k=0}^{s-1} P_k(Lx)
\ENDEQN
is independent of $k \geq 0$ and moreover for $\whx$ close to $\whp_k$ and $j\geq 1$:
\REFEQN{iterationformula}
\whH^j(\whx) = \whp_{j+k} + (\whx-\whp_k)(1+j\cdot(\whx-\whp_k)^{2m}\cdot 
\whP(\whx-\whp_k) + \OO(\whx-\whp)^{4m}).
\ENDEQN
\ENDCLAIM
Let us first see that the Claim implies $m=n$. 
Since $H$ and $\whH$ are analytically conjugate the parabolic degeneracy 
of $\whH$ at $\whp$ is also $2n$. 
However since the coefficient of the leading terms in \eqref{orderfourm} 
are non-negative and at least one of them is positive, 
it follows from \eqref{sumpoly} that the constant term of $\whP$ is positive 
and then from \eqref{iterationformula} that the degeneracy is $2m$. Therefore $m=n$.
\begin{proof}
Towards a proof of the Claim a routine computation and induction shows that 
for all $j\geq 0$ the Taylor series of $H^j$ to order $4m$ at $p_k$ is given by:
$$
H^j(x) = H^j(p_k) + (x - p_k)(1 + (x - p_k)^{2m}\cdot 
\sum_{l=0}^{j-1} P_{l+k}(x - p_k) + \OO(x - p_k)^{4m})
$$
and thus with $Q_k := (2m+1) P_k + x\cdot P'_k = Q_{s+k}$
$$
(H^j)'(x) = 1 + (x - p_k)^{2m}\cdot \sum_{l=0}^{j-1} Q_{l+k}(x - p_k) 
+ \OO(x - p_k)^{4m}.
$$
Continuing to compute $\whH'(\phi(x))$ for $x$ near $p_k$ 
starting from the first term of \eqref{whHinMd}
and using $(H^0)' \equiv 1$, $(H^j)'-1 = \OO((x-p_k)^{2m})$ 
we find
\begin{align}\label{whHprimeofphione}
\whH'(\phi(x)) &= 
\frac{\sum_{j=1}^{N} (H^j)'(x)}{\sum_{j=0}^{N-1} (H^j)'(x)}
= \frac{N + \sum_{j=1}^{N} ((H^j)'(x)-1)}{N+\sum_{j=1}^{N-1} ((H^j)'(x)-1)}\\
&= \left(1 + {\textstyle\frac{1}{N}}\sum_{j=1}^{N} ((H^j)'(x)-1)\right)
\left(1- {\textstyle\frac{1}{N}}\sum_{j=1}^{N-1} ((H^j)'(x)-1)\right) 
+ \OO(x - p_k)^{4m}\nonumber\\
&= 1 + {\textstyle\frac{1}{N}} ((H^N)'(x)-1) + \OO(x - p_k)^{4m}\nonumber\\
&= 1 + \frac{(x - p_k)^{2m}}{N}\cdot  \sum_{l=0}^{N-1} Q_{l+k}(x - p_k) 
+ \OO(x - p_k)^{4m}\nonumber
\end{align}
From the formula for $\phi$ we find the expansion of $\phi$ to order $2m$ 
at $p_k$ :
$$
\phi(x) = \whp_k+ \frac{1}{L}(x - p_k)(1 + \OO(x - p_k)^{2m})
$$
so that the expansion for $\phi^{-1}$ to order $2m$ at $\whp_k$ is:
$$
\phi^{-1}(\whx) = p_k+ L(\whx - \whp_k)(1 + \OO(\whx - \whp_k)^{2m})
$$
and thus the expansion for $\whH'$ to order $(4m-1)$ at $\whp_k$ is: 
$$
\whH'(\whx) = 1 + \frac{(L(\whx - \whp_k))^{2m}}{N}
\cdot \sum_{l=0}^{N-1} Q_{l+k}(L(\whx - \whp_k)) 
+ \OO(\whx - \whp_j)^{4m}.
$$
So by integration from $\whp_k$ we find 
$$
\whH(\whx) = \whp_{k+1} + (\whx - \whp_k)
(1 + (\whx - \whp_k)^{2m}\frac{L^{2m}}{N}
\cdot \sum_{l=0}^{N-1} P_{l+k}(L(\whx - \whp_k)) 
+ \OO(\whx - \whp_k)^{4m}),
$$
from which the Claim follows, since $N$ is a multiple of $s$ and the 
terms of the sum are repeated $N/s$ times. 
\end{proof}
\REFCLAIM{doublygoodformula}
Suppose the Taylor expansions of $H$ around 
the points $p_k$ take the form
$$
H(x) = p_{k+1} + (x - p_k)(1 + (x - p_k)^{2n}\cdot P(x - p_k) + 
(x - p_k)^{4n}\cdot R_k(x - p_k) + \OO(x - p_k)^{6n}),
$$ 
where $P$ and $R_k$ are polynomials of degree at most $(2n-1)$, 
$P$ with $P(0) > 0$ is independent of $k$ and $R_{s+k}=R_k$ for $k\geq 0$. 
Then for each $k$ the Taylor expansion of $\whH$ to order $6n$ 
at $\whp_k$ takes the form
\REFEQN{ordersixn}
\whH(\whx) = \whp_{k+1} + 
(\whx - \whp_k)(1 + (\whx - \whp_k)^{2n}\cdot \whP(\whx - \whp_k) +
(\whx - \whp_k)^{4n}\cdot \whR(\whx - \whp_k) + \OO(\whx - \whp_k)^{6n}),
\ENDEQN
where $\whR$ and $\whP(x) = L^{2n}P(Lx)$ with $\whP(0)>0$ 
are polynomials of degree at most $2n-1$ 
and are independent of the point in the orbit of $\whp = \phi(p)$.
%%$$
%%\whP(x) := L^{2n}\cdot P(Lx)
%%$$
%%and
%%\REFEQN{sumpoly}
%%\whR(x) := \frac{L^{4n}}{s}  \cdot \sum_{l=0}^{s-1} R_l(Lx).
%%\ENDEQN
\ENDCLAIM
\begin{proof}
 The proof of this Claim is similar to the proof of the first Claim, 
%%we leave the details to the reader.
%%
and we only indicate the differences.\\
For proving a formula for the $j$-th iterate the following formula is simple and useful
\REFEQN{PofHformula}
P(x(1+x^{2n}P(x))) = P(x) + x^{2n}\cdot x\cdot P'(x)\cdot P(x) + \OO(x)^{4n}
\ENDEQN
(Note that the term $x^{2n}\cdot x\cdot P'(x)\cdot P(x)$ 
contains terms of order larger than or equal $4n$, but taking them out only complicates the formula.)

By induction for each $j \geq 1$ and $x$ close to $p_k$ we find 
\begin{align*}%\label{longitformula}
H^j(x) = ~&p_{j+k} + (x - p_k)(1 + (x - p_k)^{2n}\cdot j\cdot P(x - p_k)\\
&\quad+(x - p_k)^{4n}\cdot\frac{j(j-1)}{2}((2n+1)(P(x - p_k))^2\\ 
&\qquad+ (x - p_k)P'(x - p_k)P(x - p_k))\\
&\quad\qquad+(x - p_k)^{4n}\cdot\sum_{l=0}^{j-1} R_{l+k}(x - p_k) + \OO(x - p_k)^{6n})\\
= ~&p_{j+k} + F_j(x- p_k) + (x - p_k)^{4n}\cdot\sum_{l=0}^{j-1} R_{l+k}(x - p_k) + \OO(x - p_k)^{6n})
\end{align*}
where 
\begin{align*}
F_j(x) := ~&x(1 + x^{2n}\cdot j\cdot P(x)\\
&+ x^{4n}\cdot\frac{j(j-1)}{2}((2n+1)(P(x))^2 
+ xP'(x)P(x)))
\end{align*}
is independent of $k$, i.e.~independent of the starting point in the orbit of $p$.
As above define $Q$ by the formula
$x^{2n}Q(x) := \frac{\d}{\dx}(x^{2n+1}P(x))$, and thus $Q(x) = (2n+1) P + x\cdot P'$,
and $S_k$ by the formula
$x^{4n}S_k(x) := \frac{\d}{\dx}(x^{4n+1}R_k(x))$, and thus $S_k(x) = (4n+1) R_k + x\cdot R_k' = S_{s+k}(x)$. 
Then 
\begin{align*}
(H^j))'(x) ~&= F_j'(x - p_k) + (x - p_k)^{4n}\cdot \sum_{l=0}^{j-1} S_{l+k}(x - p_k) 
+ \OO(x - p_k)^{6n}\\
~&= 1 + j (x - p_k)^{2n}Q(x - p_k) + \OO(x - p_k)^{4n}
\end{align*}
Thus
\begin{align}\label{Hjprimetwo}
\sum_{j=1}^{N-1} ((H^j)'(x)-1) &= \sum_{j=1}^{N-1} 
j (x - p_k)^{2n}Q(x - p_k) + \OO(x - p_k)^{4n}\nonumber\\
&= {\textstyle\frac{N(N-1)}{2}}(x - p_k)^{2n}Q(x - p_k) + \OO(x - p_k)^{4n}
\end{align}
Computing $\whH'(\phi(x))$ from the second formula in \eqref{whHprimeofphione} 
we obtain
\begin{align*}
\whH'(\phi(x)) &= 
\frac{N + \sum_{j=1}^{N} ((H^j)'(x)-1)}{N+\sum_{j=1}^{N-1} ((H^j)'(x)-1)}\\
&= \left(1 + {\textstyle\frac{1}{N}}\sum_{j=1}^{N} ((H^j)'(x)-1)\right)\cdot\\
&\qquad\left(1- {\textstyle\frac{1}{N}}\sum_{j=1}^{N-1} ((H^j)'(x)-1)+ 
{\textstyle\frac{(N-1)^2}{4}}(x - p_k)^{4n}(Q(x - p_k))^2 \right)\\ 
&\qquad\qquad\qquad\qquad\qquad\qquad\qquad\qquad\qquad\qquad\qquad\qquad
+ \OO(x - p_k)^{6n}\\%\nonumber\\
&= 1 + {\textstyle\frac{1}{N}} ((H^N)'(x)-1) + 
{\textstyle\frac{(N-1)^2-(N^2-1)}{4}}(x - p_k)^{4n}(Q(x - p_k))^2\\ 
%- {\textstyle\frac{(N^2-1)}{4}}(x - p_k)^{4n}(Q(x - p_k))^2 
&\qquad\qquad\qquad\qquad\qquad\qquad\qquad\qquad\qquad\qquad\qquad\qquad
+\OO(x - p_k)^{6n}\\%\nonumber\\
&= 1+ {\textstyle\frac{1}{N}}(F_N'(x-p_k) - 1) + 
\frac{(x - p_k)^{4n}}{N}\cdot  \sum_{l=0}^{N-1} S_{l+k}(x - p_k)\\ 
&\qquad\qquad\qquad\qquad\qquad\qquad
-{\textstyle\frac{1}{2}}(x - p_k)^{4n}(Q(x - p_k))^2 + \OO(x - p_k)^{6n}\\%\nonumber
\end{align*}
That is the terms of $\whH'(\phi(x))$ depending on $k$ are the terms
$$
\frac{(x - p_k)^{4n}}{N}\cdot  \sum_{l=0}^{N-1} S_{l+k}(x - p_k) + \OO(x - p_k)^{6n}
$$
of order at least $4n$.

From the definition of $\phi$ and \eqref{Hjprimetwo} we see that 
$\phi$ is independent of $k$ to order $4n$ and thus the same holds for 
$\phi^{-1}$. Combining this with the above shows that $\whH'$ is independent 
of $k$ to order $6n-1$ and thus $\whH$  is independent 
of $k$ up to and including order $6n$, as promised by the Claim.
\end{proof}


\begin{thebibliography}{13}
\bibitem[BF]{BF} B. Branner \& N. Fagella, \textit{Quasiconformal surgery in holomorphic dynamics}, Cambridge University Press, (2014).
\bibitem[C]{Cui}  G. Cui, Circle expanding maps and symmetric structures. \textit{Ergodic Theory Dynam. Systems} 18
(1998), no. 4, 831--842
\bibitem[DH]{DH} A. Douady \& J. H. Hubbard, On the dynamics of polynomial-like mappings, \textit{Ann. Sci. \'{E}cole
Norm.
Sup.},(4), Vol.$18$ ($1985$), 287--343.
\bibitem[L]{L}  L. Lomonaco. \textit{Parabolic-like maps}. Erg. Theory and Dyn. Syst. (first view, published on-line the 3rd of July 2014), 1--27.
\bibitem[MJ]{MJ}
J. Ma. \textit{On evolution of a class of Markov maps.} Undergraduate thesis (in Chinese), University of Science and Technology of China, 2007.
\bibitem[M]{Mane}
R. Ma\~n\'e. \textit{Hyperbolicity, sinks and measure in one-dimensional dynamics.} Comm. Math. Phy. Vol. 100 (1985), 495-524.
\bibitem[MMS]{MMS}
M. Martens, W. de Melo \& S. van Strien. \textit{Julia-Fatou-Sullivan theory for real one-dimensional dynamics.}
Acta Math. 168 (1992), 273-318.
\bibitem[dMvS]{dMvS} W. de Melo \& S. van Strien, \textit{One-Dimensional Dynamics}, Springer-Verlag, (1993).
\bibitem[R]{R} W. Rudin, \textit{Real and Complex Analysis}, McGraw Hill, (1966), (1974).
\bibitem [S]{Sh} M. Shishikura, Bifurcation of parabolic fixed points, \textit{The Mandelbrot set, Theme and
Variations, } (325-363), \textit{London Math. Soc. Lecture Note Ser., 274} Cambridge Univ. Press, (2000).
\end{thebibliography}
\end{document}